\documentclass{article}
\usepackage{graphicx}
\usepackage{epstopdf}
\usepackage{amsmath}
\usepackage{amssymb}
\usepackage{bm}
\usepackage{bbm}
\usepackage{enumerate}
\usepackage[T1]{fontenc}
\usepackage[latin1]{inputenc}
\usepackage{hyperref}
\usepackage{color}
\usepackage{curves}
\usepackage{tikz}
\usepackage[hang]{caption}
\usepackage{ntheorem}
\usepackage{array,float,caption,nicefrac}
\usepackage{mathrsfs}

\title{Left-continuous random walk on $\IZ$ and the parity of its hitting times}
\author{Timo Vilkas
		\\\normalsize Lunds Universitet
}

\theoremstyle{break}
\newtheorem{theorem}{Theorem}[section]
\newtheorem{corollary}{Corollary}[section]
\newtheorem{lemma}{Lemma}[section]

\theorembodyfont{\upshape}
\newtheorem{definition}{Definition}
\newtheorem*{remark}{Remark}
\newtheorem*{example}{Example}

\makeatletter
\let\c@proposition\c@theorem
\let\c@lemma\c@theorem
\let\c@corollary\c@theorem
\makeatother

\newenvironment{proof}{\noindent{\sc Proof:}}{\vspace{-1em}~\hfill $\square$\vspace{2em}}

\newcommand\IN{\mathbb{N}}
\newcommand\IR{\mathbb{R}}
\newcommand\IZ{\mathbb{Z}}

\newcommand\IE{\mathbb{E}\,}
\newcommand\Prob{\mathbb{P}}
\renewcommand\epsilon{\varepsilon}
\renewcommand\phi{\varphi}

\definecolor{darkblue}{rgb}{0,0,.5}
\hypersetup{colorlinks=true, breaklinks=true, linkcolor=blue, 
citecolor=darkblue, menucolor=blue, urlcolor=blue}

\begin{document}
\newpage
\maketitle
%%%%%%%%%%%%%%%%%%%%%%%%%%%
% abstract, keywords and Subject classification are optional.
%%%%%%%%%%%%%%%%%%%%%%%%%%%
\begin{abstract}
When it comes to random walk on the integers $\IZ$, the arguably first step of generalization beyond simple random
walk is the class of one-sidedly continuous random walk, where the stepsize in only one direction is bounded by 1. Moreover, the time until state 0 is hit by left-continuous random walk on $\IZ$ has a direct connection to the total progeny in branching processes. In this article, the probability of left-continuous random walk to be negative at an even (resp.\ odd) time is derived and used to determine the probability of nearly left-continuous random walk to eventually become negative.
\end{abstract}

% Most people don't use these, so they are "commented out"
% by starting the lines with a "%"

%%%%%%%%%%%%%%%%%%%%%%
% % Here is the start of the Text
%%%%%%%%%%%%%%%%%%%%%%
\section{Introduction}
In 1874, Francis Galton and Henry William Watson published a paper titled ``On the probability of the extinction of families'', laying the foundation of what is nowadays textbook material of basic courses in probability. Independently also
Ir\'en\'ee-Jules Bienaym\'e derived and analyzed the following branching process:
Consider a collection $(\xi_{i,j})_{i,j\in\IN_0}$ of independent copies of $\xi$, a non-negative, integer-valued random variable. The process starts with one progenitor who gives rise to $\xi_{0,1}$ children. Each individual in this
first generation then independently gives rise to $\xi_{1,j}$ individuals, $1\leq j \leq \xi_{0,1}$. This branching continues in all subsequent generations in the same manner.

The original question (Galton and Watson were concerned with the extinction of family names) addresses the probability of this branching process to die out, i.e.\ to determine $\rho=\Prob(Z_i=0\text{ for some }i\in\IN)$ where $Z_i$ defined via
\[Z_0=1\quad \text{and}\quad Z_i=\sum_{j=1}^{Z_{i-1}}\xi_{i,j}\quad\text{for }i\in\IN\]
denotes the number of individuals in the $i$th generation.
The fundamental \emph{extinction criterion} (cf.\ Chapter 1 in \cite{Athreya-Ney} for instance) states that for a non-deterministic distribution $\xi$, it holds $\rho<1$ iff $\IE\xi >1$. Another way to keep an account of the total progeny is to explore the family tree individual by individual and count the unexplored offspring, i.e.\ adding $\xi_{i,j}-1$ to the current count when individual $j$ in generation $i$ is processed. Given the properties of $\xi$ this random walk on the integers is what is commonly referred to as left-continuous. 

\begin{definition}
	Let $(S_n)_{n\in\IN_0}$ be a random walk on the integers. If its i.i.d.\ increments $X_n=S_n-S_{n-1}$
	are bounded from below by $-1$, i.e.\ $\Prob(X_n\geq -1)=1$, the walk is called {\em left-continuous} or {\em skip-free to the left}. Let us write $(p_k)_{k\geq -1}$ for the probability mass function and
	\[g(x)=\sum_{k=-1}^\infty p_k\,x^k=\frac{p_{-1}}{x}+p_0+p_1\,x+p_2\,x^2+\dots\]
	for the probability generating function of the corresponding increment distribution.
\end{definition}
Note that $g$ is not defined in $x=0$ in case $p_{-1}>0$ and that the above sum converges (absolutely) on $[-1,0)\cup(0,1]$ irrespectively of $(p_k)_{k\geq -1}$.

The connection between skip-free random walk and the progeny in the Galton-Watson-Bienaym\'e branching process was
established in the 1960s, cf.\ \cite{GWT}, and one of the reasons why the former received attention. Using this connection, we can rewrite the extinction probability as
\[\rho=\Prob(S_n =0\text{ for some }n\in\IN\,|\,S_n=1)=\Prob(S_n =-1\text{ for some }n\in\IN\,|\,S_n=0),\] where $(S_n)_{n\in\IN_0}$ is the left-continuous walk on $\IZ$ with increments distributed as $\xi-1$. According to the extinction criterion, $\rho$ is less than 1 iff $\IE(\xi-1)>0$, ignoring the trivial case $\xi\equiv 1$. This corresponds to the random walk having a positive drift.

In a different context, $S_n$ can be seen to describe the credit of a gambler after round $n$ in a game, where each round one token is at stake and the yield distributed as $\xi$. $\rho$ can then be interpreted as \emph{ruin probability}.
\vspace{1em}

Left-continuous random walk is a prominent example in standard references on the matter (e.g.\ \cite{Spitzer}) as it maintains a substantial part of the tractability of simple random walk without being quite that simplistic. While \cite{hitting} settles the question to what extent knowing $\Prob(S_n =0)$ for all $n\in\IN$ determines the increment distribution of left-continuous random walk on $\IZ$, \cite{skipfree2} focusses on the limiting behavior of the walk on $\IN_0$ where state 0 is chosen to be absorbing. With simple martingale methods, the probability of left-continuous random walk to visit a given level below the starting state can be calculated (cf.\ Thm.\ 5.16 in \cite{Durrett}). We chose to include a different approach, see Lemma \ref{rhorec} below. The hitting time theorem for left-continuous random walk, which states that given $\{S_n=0\}$ the probability that $n$ is the first time the walk started at $k$ visits $0$ equals $\frac{k}{n}$, irrespectively of the increment distribution, was reproved in \cite{hitting2} using basic yet elaborate combinatorial arguments. In the article by Brown, Pek\"oz and Ross \cite{skipfree} mean and variance of the minimum of
left-continuous random walk on $\IZ$ are derived using the fact that the walk with positive drift conditioned to return to state $0$ (or below) behaves like an unconditioned walk with a different increment distribution, as widely known and used for simple random walk.

This paper is concerned with the two events of left-continuous random walk on $\IZ$ to eventually be negative at an \emph{even} resp.\ \emph{odd} time and with deriving their probabilities. In the short next section, we warm up by proving that random walk with positive drift stays non-negative with positive probability. Section \ref{sec3} contains the definition of all relevant hitting times and related probabilities as well as useful relations between the latter and the probability generating function of the increment distribution. In the last section, the central Theorem \ref{even} as well as some more or less immediate corollaries are proved. Finally, these results about parity of hitting times for left-continuous random walk on $\IZ$ are used to determine the probability to become negative even for random walk that is not quite left-continuous, cf.\ Corollary \ref{asf}. 

\section{Random walk with positive drift}\label{sec1}

Let us start with a short proof of the well-known fact that random walk on $\IZ$ with positive drift has a positive probability to never hit $\IZ\setminus\IN_0=\{\dots,-2,-1\}$, if started at $0$.

\begin{lemma}\label{drift}
	Let $(X_n)_{n\in\mathbb{N}}$ be an i.i.d.\ sequence with $\IE(X_1)>0$.
	Then the infimum of the random walk $(S_n)_{n\in\IN_0}$, defined by $S_0=0,\ S_{n+1}=S_n+X_{n+1}$, is attained at $S_0=0$ with positive probability.
\end{lemma}
\begin{proof}
	By the law of large numbers it holds almost surely that \[\lim_{n\to\infty}\frac{S_n}{n} =\lim_{n\to\infty}\frac1n \sum_{k=1}^n X_k=\mathbb{E}(X_1)>0.\] Hence, there exists $n_0\in\mathbb{N}$ such that $\mathbb{P}\big(S_n \geq 1\ \text{ for all } n\geq n_0\big)>0$.
	Since $n_0$ is finite, there must be some $0\leq n_1< n_0$ such that
	\begin{equation*}
		\mathbb{P}\Big(S_{n_1} =\min\big\{S_n,\; n\in\mathbb{N}_0\big\}\Big)>0.
	\end{equation*}
	
	As a consequence, the event $\{\sum_{k=n_1+1}^n X_k\geq 0$ for all $n>n_1\}$ must have strictly positive probability. Due to the fact that $(X_n)_{n>n_1}$ and $(X_n)_{n\in\mathbb{N}}$ have the same distribution, the claim follows.
\end{proof}

Note that the same proof in fact also works for more general walks on $\IR$ with positive drift, which is why the restriction of $X_1$ to be integer-valued was omitted in the statement of the lemma. For the remainder of the paper, we will however only be concerned with random walks on the integers, i.e.\ increment distributions with $\Prob(X_1\in\IZ)=1$.

\section{Hitting times and their parity}\label{sec3}

In this section, we want to determine the probabilities of different hitting times to be finite (and odd) for a left-continuous random walk $(S_n)_{n\in\IN_0}$, defined by $S_0=0,\ S_{n+1}=S_n+X_{n+1}$. To make things non-trivial, we consider walks with positive drift that are not monotone, i.e.\ $\IE(X_1)>0$ and $p_{-1}>0$. To this end, let
\begin{align*}T_{-1}&=\inf\{n,\;S_n=-1\}\quad\text{and}\\
	T_0^+&=\inf\{n>0,\;S_n=0\}
\end{align*}
denote the hitting time of state $-1$ and the return time to state $0$ respectively. Further we define the
following probabilities:
\begin{gather}\label{probs}
	\begin{aligned}
		\rho&=\Prob(T_{-1}<\infty)\\
		\rho_\mathrm{odd}&=\Prob(T_{-1}\text{ is odd}\,|\,T_{-1}<\infty)\\
		\sigma&=\Prob(S_n>0\text{ for all }n\in\IN)\\
		\tau&=\Prob(T_0^+<\infty,\ T_0^+<T_{-1})\\
		\tau_\mathrm{odd}&=\Prob(T_0^+\text{ is odd}\,|\,T_0^+<\infty,\,T_0^+<T_{-1}).
	\end{aligned}
\end{gather}
In words, $\rho$ denotes the probability, that state $-1$ will be visited, $\rho_\mathrm{odd}$ the conditional probability that given the walk reaches state $-1$, it will do so for the first time after an odd number of steps; $\sigma$ is the probability that the walk stays strictly positive after its start in state $0$,  and $\tau$ is the probability that $0$ is revisited, before $-1$ is visited in case the latter happens. Finally, $\tau_\mathrm{odd}$ is the conditional probability that, given a return to $0$ before visiting $-1$, this positive excursion will take an odd number of steps. For ease of notation we write $\tau=\Prob(T_0^+<T_{-1})$ etc.\ as the strict inequality forces $T_0^+$ to be finite.

For left-continuous walks, one can easily derive an equation for the probability $\rho$ that the walk visits $-1$ after start at $0$, cf.\ Lemma 2 in \cite{skipfree}:

\begin{lemma}\label{rhorec}
	Let $(S_n)_{n\in\IN_0}$ be a left-continuous, non-monotonic random walk with positive drift, i.e.\ $\Prob(X_1\geq -1)=1>\Prob(X_1\geq 0)$ and $\IE(X_1)>0$.
	Then the probability $\rho$, that state $-1$ is visited after start at $0$, is given by the unique solution to $g(x)=1$ in the open interval $(0,1)$.
\end{lemma}
\begin{proof}
	By the definition of $\rho$ and the strong Markov property it immediately follows that $\Prob(T_{-1}<\infty\,|\,S_0=k)=\rho^{k+1}$, for all $k\in\IN_0$. Consequently, conditioning on the first step gives
	\begin{equation}\label{eq:rhorec}
		\rho=\sum_{k=-1}^\infty p_k\,\rho^{k+1}=\rho\cdot g(\rho).
	\end{equation}
	Lemma \ref{drift} (together with $\rho\geq p_{-1}>0)$ implies $0<\rho<1$ and the claim follows from the strict convexity of $g$ on $(0,1)$ and $g(1)=1$.
\end{proof}

Observe that in the setting of the above lemma, from \eqref{eq:rhorec} one can in fact deduce $\rho>p_{-1}$, as
$p_k>0$ for some $k>0$ since $0<\IE(X_1)<\sum_{k>0}p_k\cdot k$.
Next, we want to address the parity of the hitting times $T_{-1}$ and $T_0^+$.

\begin{lemma}\label{calc}
	Let $(S_n)_{n\in\IN_0}$ be a left-continuous, non-monotonic random walk with positive drift, started in $S_0=0$. 
	Then the following hold:
	\begin{enumerate}[(a)]
		\item $p_{-1}+\tau+\sigma=1$
		\item $\sigma=p_{-1}\cdot \frac{1-\rho}{\rho}$ and $\tau=1-\frac{p_{-1}}{\rho}$
	\end{enumerate}
\end{lemma}

\begin{proof}
	Part (a) follows immediately from left-continuity of the walk, by case differentiation whether the first step is down or state $0$ revisited if it is not. As to part (b), conditioning on the first step and using translation invariance of the walk in the form $\Prob(T_0^+<\infty\,|\,S_0=k)=\Prob(T_{-1}<\infty\,|\,S_0=k-1)=\rho^{k}$, for all $k>0$, gives:
	\[\sigma=\sum_{k=1}^\infty p_k\,\big(1-\rho^k\big)=\sum_{k=0}^\infty p_k\,\big(1-\rho^k\big)=\big(1-p_{-1}\big)-\Big(g(\rho)-\tfrac{p_{-1}}{\rho}\Big),\]
	from which the first claim follows (using $g(\rho)=1$, see Lemma \ref{rhorec} above). The second one is now either a simple consequence of part (a) or derived in the same vein:
	\[\tau=\sum_{k=0}^\infty p_k\,\rho^k=g(\rho)-\tfrac{p_{-1}}{\rho}=1-\tfrac{p_{-1}}{\rho}.\]
\end{proof}

\begin{theorem}\label{skipfree}
	Let $(S_n)_{n\in\IN_0}$ be a left-continuous, non-monotonic random walk with positive drift, started in $S_0=0$ and
	$\rho_\mathrm{odd},\tau_\mathrm{odd}$ be the corresponding conditional probabilities defined in \eqref{probs}. Then the following relations hold:
	\begin{enumerate}[(a)]
		\item $\rho_\mathrm{odd}>\frac12$ and $g\big(\rho\,(1-2\rho_\mathrm{odd})\big)=-1$
		\item $\tau_\mathrm{odd}=\frac{p_{-1}\,(1-\rho_\mathrm{odd})}{\rho\tau\,(2\rho_\mathrm{odd}-1)}$
	\end{enumerate}
\end{theorem}
\begin{proof}
	To prove the first part, let us begin with the elementary calculation of the probability that a binomial variable
	takes on an even value: Let $Y\sim\mathrm{Bin}(n,p)$, then 
	\begin{align*}\Prob(Y\text{ is even})&=\sum_{k=0}^{\lfloor\frac n2\rfloor}\binom{n}{2k}p^{2k}\,(1-p)^{n-2k}\\
		&=\frac12\Bigg(\sum_{k=0}^n\binom{n}{k}p^{k}\,(1-p)^{n-k}+\sum_{k=0}^n\binom{n}{k}(-p)^k\,(1-p)^{n-k}\Bigg)\\
		&=\tfrac12\,\big(1+(1-2p)^n\big).
	\end{align*}
	Assume for now that the random walk starts in $S_0=k\geq0$ and let $T_{-1}(k)$ stand for $T_{-1}\,|\,S_0=k$. Due to its left-continuity, the time until state $-1$ is visited can then be written as sum of the (possibly infinite) times \[T_{i\to i-1}=\inf\{n>0,\;S_n=i-1\}-\inf\{n>0,\;S_n=i\},\quad\text{for }0\leq i\leq k,\] it takes to reach level $i-1$ after the first time level $i$ was visited, i.e.
	\[T_{-1}(k)=\sum_{i=0}^k T_{i\to i-1}\quad\text{and}\quad\big\{T_{-1}(k)<\infty\big\}=\bigcap_{i=0}^k \{T_{i\to i-1}<\infty\},\]
	see Figure \ref{posexc} for an illustration. By the strong Markov property and translation invariance of the walk, the random variables $T_{k\to k-1},\dots,T_{0\to-1}$ are independent and distributed as $T_{-1}(0)$, the time to hit state $-1$ after a start in $S_0=0$.
	\begin{figure}[h]
		\includegraphics[width=\textwidth]{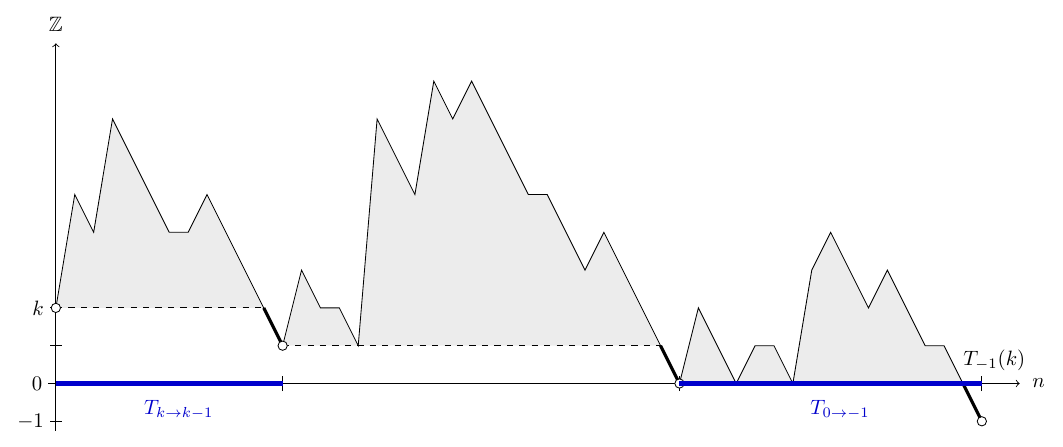}
		\caption{Starting in $S_0=k\geq0$, the total time $T_{-1}(k)$ until state $-1$ is hit is the sum of $k+1$ independent copies of $T_{-1}(0)$.\label{posexc}}
	\end{figure}
	
	Due to product structure of the set $\{T_{-1}(k)<\infty\}$ it further holds that $T_{-1}(k)$ conditioned to be finite is distributed as the sum of $k+1$ i.i.d.\ copies of $T_{-1}(0)$ conditioned to be finite. At this point, it is crucial to observe that $T_{-1}(k)$ is even, if and only if an even number of these $k+1$ copies are odd, and that given $\{T_{-1}(k)<\infty\}$ this number is binomially distributed with parameters $k+1$ and $\rho_\mathrm{odd}$ by (conditional) independence.\vspace{1em}
	
	Returning to the walk starting in $S_0=0$, we can condition on the first step once again. In doing so, two things are important to note: On the one hand, given $X_1=k\geq0$, the subsequent walk behaves like one started in $S_0=k$, besides the fact that this first step flips the parity of $T_{-1}$. On the other hand, conditioning on $\{T_{-1}<\infty\}$ also changes the distribution of the first increment $X_1$. Applying Bayes' theorem, we find for all $k\geq -1$:
	\[\Prob(X_1=k\,|\,T_{-1}<\infty)=\frac{\Prob(X_1=k)}{\Prob(T_{-1}<\infty)}\cdot\Prob(T_{-1}<\infty\,|\,X_1=k)
	=\frac{p_k}{\rho}\rho^{k+1}=p_k\,\rho^k.\]
	With the above reasoning, we established the following recursion for $\rho_\mathrm{odd}$:
	\begin{align*}
		\rho_\mathrm{odd}&=\sum_{k=-1}^\infty \Prob(X_1=k\,|\,T_{-1}<\infty)\cdot\Prob(T_{-1}\text{ is even}\,|\,S_0=k,\ T_{-1}<\infty)\\
		&=\sum_{k=-1}^\infty p_k\,\rho^k\cdot\frac12\,\Big(1+(1-2\rho_\mathrm{odd})^{k+1}\Big)\\
		&=\frac12 \Big(g(\rho)+(1-2\rho_\mathrm{odd})\cdot g\big(\rho\,(1-2\rho_\mathrm{odd})\big)\Big).
	\end{align*}
	Here we tacitly used the fact that $\rho\,(1-2\rho_\mathrm{odd})\neq0$, i.e.\ that $g$ is defined for this argument.
	Assuming $\rho_\mathrm{odd}\leq\frac12$ in the second equation, however, would lead to
	\[\rho_\mathrm{odd}\geq \frac{p_{-1}}{\rho}+\sum_{k=0}^\infty p_k\,\rho^k\cdot\frac12\,(1+0)=\frac{p_{-1}}{2\rho}+\frac{g(\rho)}{2}>\frac12,\]
	i.e.\ a contradiction. With help of a small algebraic manipulation (using $g(\rho)=1$ and $(1-2\rho_\mathrm{odd})\neq0$) we
	finally arrive at $1+g\big(\rho\,(1-2\rho_\mathrm{odd})\big)=0$.
	
	Very much like for $\rho_\mathrm{odd}$ we can calculate $\tau_\mathrm{odd}$ by conditioning on the first step, however, this time with $\Prob(X_1=-1\,|\,T_0^+<T_{-1})=0$ and
	\[\Prob(X_1=k\,|\,T_0^+<T_{-1})=\Prob(T_0^+<\infty\,|\,X_1=k)\cdot\frac{\Prob(X_1=k)}{\Prob(T_0^+<T_{-1})}\\=\rho^{k}\cdot\frac{p_k}{\tau},\]
	for all $k\geq0$. The first equation here follows (besides using Bayes) from the fact that for $k\geq 0$ it holds $\{X_1=k,\ T_0^+<T_{-1}\}=\{X_1=k,\ T_0^+<\infty\}$, by left-continuity of the walk.
	
	Using translation invariance of the walk, we arrive at
	\begin{align*}\tau_\mathrm{odd}&=\Prob(T_0^+\text{ is odd}\,|\,T_0^+<T_{-1})\\
		&=\sum_{k=0}^\infty \Prob(X_1=k\,|\,T_0^+<T_{-1})\cdot \Prob(T_0^+\text{ is odd}\,|\,T_0^+<\infty,\, X_1=k)\\
		&=\sum_{k=0}^\infty \rho^{k}\cdot\frac{p_k}{\tau}\cdot \Prob(T_{-1}\text{ is even}\,|\,T_{-1}<\infty,\, S_0=k-1)\\
		&=\sum_{k=0}^\infty \rho^{k}\cdot\frac{p_k}{\tau}\cdot\frac12\,\Big(1+(1-2\rho_\mathrm{odd})^{k}\Big)\\
		&=\frac{1}{2\tau}\Bigg(g(\rho)-\frac{p_{-1}}{\rho}+g\Big(\rho\,(1-2\rho_\mathrm{odd})\Big)-\frac{p_{-1}}{\rho\,(1-2\rho_\mathrm{odd})}\Bigg)\\
		&=\frac{p_{-1}\,(1-\rho_\mathrm{odd})}{\rho\tau\,(2\rho_\mathrm{odd}-1)}.
	\end{align*}
	Note that $\rho\in(p_{-1},1)$, thus $\tau>0$, and $\rho_\mathrm{odd}>\frac12$, which gives $\rho\tau\,(2\rho_\mathrm{odd}-1)>0$. Besides that, both $g(\rho)=1$ and part (a) were used in the last step.
\end{proof}

\begin{example}
	Consider increments of the form $X_n=Y_n-1$, where $(Y_n)_{n\in\IN}$ is an i.i.d.\ sequence of random variables that are Poisson distributed with parameter $\lambda>1$, i.e.\ \[p_k:=\Prob(X_1=k)=\mathrm{e}^{-\lambda}\frac{\lambda^{k+1}}{(k+1)!},\text{ for all }k\geq -1.\] Then $p_{-1}=\mathrm{e}^{-\lambda}$ and $g(x)=\frac1x\,\mathrm{e}^{\lambda(x-1)}$.
	
	Choosing a numerical value for the parameter $\lambda$ one can calculate the relevant (conditional) probabilities using Lemma \ref{calc} and Theorem \ref{skipfree}.
	With $\lambda=\frac32$ for instance, the corresponding computations give (approximately):
	$\rho=0.417188$, $\sigma=0.311713,$ $\tau=0.465157$ as well as $\rho_\mathrm{odd}=0.706513$ and $\tau_\mathrm{odd}
	=0.817032.$
\end{example}

\section{Using parity to deal with nearly left-continuous walks on $\IZ$ with separable increments}\label{separable}

Given $(S_n)_{n\in\IN_0}$ as above, we can use Lemma \ref{calc} and Theorem \ref{skipfree} to define and analyze an auxiliary Markov chain, designed to determine the probability of the event $E=\{S_{2n}<0 \textup{ for some }n\in\IN\}$, i.e.\ that the walk is negative at an even time.

\begin{theorem}\label{even}
	Let $(S_n)_{n\in\IN_0}$ be a left-continuous, non-monotonic random walk on $\IZ$ with positive drift and
	$\sigma$, $\rho$ and $\tau$ as well as $\rho_\mathrm{odd}$ and $\tau_\mathrm{odd}$ be the corresponding (conditional) probabilities defined in \eqref{probs}. Then it holds
	\begin{equation}\label{evenprob}
		\Prob(E\,|\,S_0=0)= \rho\,\bigg(1-\frac{\sigma\cdot\rho_\mathrm{odd}}{1-\tau\cdot(1-\tau_\mathrm{odd})}\bigg).
	\end{equation}
\end{theorem}
\begin{proof}
	As already mentioned, to calculate the probability of the event $E$, that $(S_n)_{n\in\IN_0}$ visits $\IZ\setminus\IN_0=\{\dots,-2,-1\}$ at an even time, the results from Section \ref{sec3} come in useful.
	Let us define a Markov chain, by disecting the trajectory of $(S_n)_{n\in\IN_0}$ into the starting point $S_0=0$ and segments leading
	to visits in $\IZ\setminus\IN_0$, keeping track of the parity of time. The states ``no further visit to $-1$'' and
	``$S_n< 0$ for $n$ even'' are considered to be absorbing. However, the state ``$S_n=-1$ for $n$ odd'' is transient, as it does not give an answer to the question whether $E$ eventually occurs or not.
	Representing the relevant segments by nodes as depicted below in Figure \ref{MC}, we get a finite absorbing Markov chain with transition probabilities as shown, cf.\ \eqref{probs}.
	
	\begin{figure}[h]
		\vspace{-0.8cm}\begin{center}
			\includegraphics[width=0.8\textwidth]{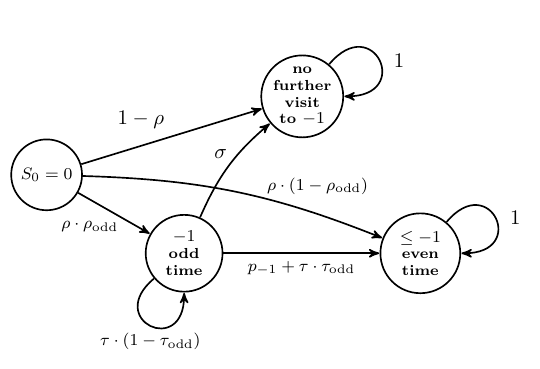}\end{center}\vspace{-0.5cm}
		\caption{Simple MC tracking the parity of time when $(\tilde{S}_n)_{n\in\IN_0}$ becomes negative.\label{MC}}
	\end{figure}
	The event that the smallest $n$ for which $S_n=-1$ is (finite and) odd, for instance, corresponds to $\{T_{-1}\text{ is odd}\}$, which has probability  $\rho\cdot\rho_\mathrm{odd}$. Given that we are in state $-1$ at an odd time, the probability that the trajectory takes a positive excursion of even length (to be in $-1$ again at an odd time next) is $\tau$ (for taking a positive excursion from $-1$ to $-1$) times $1-\tau_\mathrm{odd}$, the conditional probability that, given the walk returns to $-1$ before visiting $-2$, this excursion consists of an even number of steps. Note that the Markov property is inherited from the random walk $(S_n)_{n\in\IN_0}$ due to the fact that the two transient states of the chain correspond to states (with time stamp) of the walk. 
	
	Ordering the states from left to right, we arrive at the following transition matrix:
	\[P=\begin{bmatrix} 0 & \rho\cdot\rho_\mathrm{odd}& 1-\rho &\rho\,(1-\rho_\mathrm{odd})\\
		0 & \tau\,(1-\tau_\mathrm{odd}) &\sigma & p_{-1}+\tau\cdot\tau_\mathrm{odd}\\
		0&0&1&0\\
		0&0&0&1
	\end{bmatrix}.\]
	Given that the chain starts in the first state ``$S_0=0$'', the probability to get absorbed in the fourth one ``$S_n\leq -1$ for $n$ even''
	is given by element $b_{12}$ of the matrix $B=NR\in\IR^{2\times2}$, where $N=(\mathbf{I}-Q)^{-1}$ is the so-called fundamental matrix, $\mathbf{I}$ refers to the $2\times 2$ identity matrix and $Q,R$ are the submatrices the canonical form of $P$ consists of:
	\[P=\begin{bmatrix} Q & R\\
		\mathbf{0} & \mathbf{I}\\
	\end{bmatrix},\text{ i.e.}\ \
	Q=\begin{bmatrix} 0 & \rho\cdot\rho_\mathrm{odd}\\
		0 & \tau\cdot(1-\tau_\mathrm{odd})\\
	\end{bmatrix},\ \ 
	R=\begin{bmatrix} 1-\rho &\rho\cdot(1-\rho_\mathrm{odd})\\
		\sigma & p_{-1}+\tau\cdot\tau_\mathrm{odd}\\
	\end{bmatrix}.\]
	The claim \eqref{evenprob} now follows by some basic matrix calculations. For details see Chapter 11.2 in \cite{Markov}.
\end{proof}

\begin{corollary}\label{odd}
	Letting $O=\{S_{2n-1}<0 \textup{ for some }n\in\IN\}$ denote the event that the walk $(S_n)_{n\in\IN_0}$, defined as above, is negative at an odd time instead, we can conclude
	\[\Prob(O\,|\,S_0=0)= \rho\,\bigg(1-\frac{\sigma\cdot(1-\rho_\mathrm{odd})}{1-\tau\cdot(1-\tau_\mathrm{odd})}\bigg).\]
	and
	\[\Prob(E\cap O\,|\,S_0=0)= \rho\,\bigg(1-\frac{\sigma}{1-\tau\cdot(1-\tau_\mathrm{odd})}\bigg)=\rho\,\frac{p_{-1}+\tau\cdot\tau_\mathrm{odd}}{1-\tau\cdot(1-\tau_\mathrm{odd})}.\]
\end{corollary}

\begin{proof}
	To determine $\Prob(O\,|\,S_0=0)$ we can adapt the approach above by essentially
	letting the states ``$S_n=-1$ for $n$ odd'' and ``$S_n\leq-1$ for $n$ even'' change their roles, cf.\ Figure \ref{MC1}.
	\begin{figure}[h!]
		\vspace{-1cm}\begin{center}
			\includegraphics[width=0.8\textwidth]{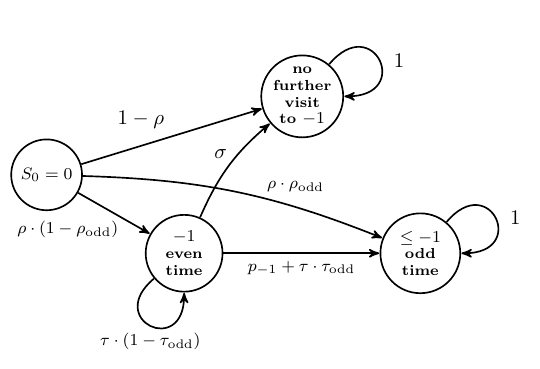}\end{center}\vspace{-0.5cm}
		\caption{Slightly altered Markov chain construction\label{MC1}}
	\end{figure}
	
	\noindent
	Given the slightly changed transition matrix ($\rho_\mathrm{odd}$ and $1-\rho_\mathrm{odd}$ get swapped) and
	\[	Q=\begin{bmatrix} 0 & \rho\cdot(1-\rho_\mathrm{odd})\\
		0 & \tau\cdot(1-\tau_\mathrm{odd})\\
	\end{bmatrix},\ \ 
	R=\begin{bmatrix} 1-\rho &\rho\cdot\rho_\mathrm{odd}\\
		\sigma & p_{-1}+\tau\cdot\tau_\mathrm{odd}\\
	\end{bmatrix}\]
	following the above reasoning verbatim, element $b_{12}$ of the matrix $B=(\mathbf{I}-Q)^{-1}R$ corresponds to the probability we were looking for. Using 
	\begin{align*}\rho&=\Prob(T_{-1}<\infty\,|\,S_0=0)=\Prob(E\cup O\,|\,S_0=0)\\
		&=\Prob(E\,|\,S_0=0)+\Prob(O\,|\,S_0=0)-\Prob(E\cap O\,|\,S_0=0)
	\end{align*}
	and Lemma \ref{calc}(a), we arrive at the second claim.
\end{proof}

In contrast to the probability of hitting $\IZ\setminus\IN_0$, i.e.\ the event $\{T_{-1}<\infty\}$, which simply becomes $\rho^{k+1}$ for a left-continuous random walk on $\IZ$
started in general $k\in\IN_0$, the probability of the event $E=\{S_{2n}<0 \textup{ for some }n\in\IN\}$ given a start at $k>0$ is not just a power of $\Prob(E\,|\,S_0=0)$. However, the construction above can be used to determine this probability even for $S_0>0$:

\begin{corollary}\label{generalstart}
	For general $k\in\IN_0$ it holds
	\[\Prob(E\,|\,S_0=k)= \rho^{k+1}\,\bigg(1-\frac{\sigma}{2}\cdot\frac{1-(1-2\rho_\mathrm{odd})^{k+1}}{1-\tau\cdot(1-\tau_\mathrm{odd})}\bigg),\]
	\[\Prob(O\,|\,S_0=k)= \rho^{k+1}\,\bigg(1-\frac{\sigma}{2}\cdot\frac{1+(1-2\rho_\mathrm{odd})^{k+1}}{1-\tau\cdot(1-\tau_\mathrm{odd})}\bigg)\quad\text{and}\]
	\[\Prob(E\cap O\,|\,S_0=k)= \rho^{k+1}\,\bigg(1-\frac{\sigma}{1-\tau\cdot(1-\tau_\mathrm{odd})}\bigg).\]
\end{corollary}
\begin{proof}
	Using the same construction as above, we end up with the Markov chain and corresponding transition probabilities as depicted in Figure \ref{MC2}.
	
	\begin{figure}[h!]
		\vspace{-0.8cm}\begin{center}
			\includegraphics[width=\textwidth]{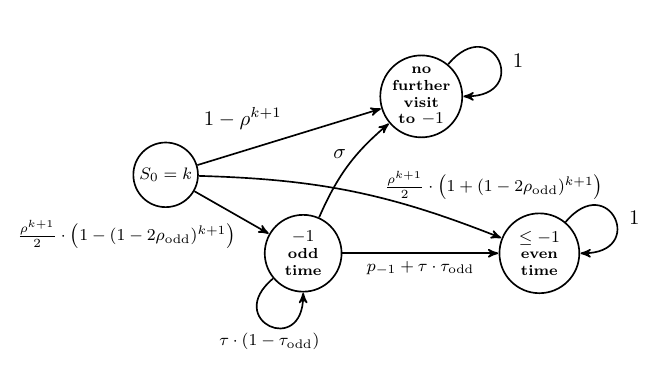}\end{center}\vspace{-0.5cm}
		\caption{A different starting value $S_0$ only changes the transition probabilities of the Markov chain from this state.\label{MC2}}
	\end{figure}
	
	\noindent
	This leads to
	\[P=\begin{bmatrix} 0 & \frac{\rho^{k+1}}{2}\big(1-(1-2\rho_\mathrm{odd})^{k+1}\big)& 1-\rho^{k+1} &\frac{\rho^{k+1}}{2}\big(1+(1-2\rho_\mathrm{odd})^{k+1}\big)\\
		0 & \tau(1-\tau_\mathrm{odd}) &\sigma & p_{-1}+\tau\cdot\tau_\mathrm{odd}\\
		0&0&1&0\\
		0&0&0&1
	\end{bmatrix},\]
	hence \[
	Q=\begin{bmatrix} 0 &\rho^{k+1}\,(1-q)\\
		0 & \tau(1-\tau_\mathrm{odd})\\
	\end{bmatrix}\quad\text{and}\quad
	R=\begin{bmatrix} 1-\rho^{k+1} &\rho^{k+1}\,q\\
		\sigma & p_{-1}+\tau\cdot\tau_\mathrm{odd}\\
	\end{bmatrix},\]
	where $q:=\tfrac12\,(1+(1-2\rho_\mathrm{odd})^{k+1})$ denotes the probability of a random variable, which is binomially distributed with parameters $k+1$ and $\rho_\mathrm{odd}$, to be even.
	The second element in the first row of matrix $B=(\mathbf{I}-Q)^{-1}R$ finally gives the probability of interest as before.
	
	Changing parities translates to swapping $q$ and $1-q$ in the construction and calculations above, which gives the second claim. Finally, the probability for $E\cap O$ follows immediately by the inclusion-exclusion principle.
\end{proof}

To extend the analysis of random walk on $\IZ$ to increment distributions concentrated on $\{-2,-1\}\cup\IN_0$, rather than $\{-1\}\cup\IN_0$ as for left-continuous walks, we consider the special case in which the increment distribution arises as the convolution of two i.i.d.\ integer-valued random variables:
\begin{definition}
	Let $X$ be an integer-valued random variable. Let its distribution be called {\em separable}, if there exist independent and identically distributed, integer-valued random variables $Y_1$ and $Y_2$, such that 
	\begin{equation*}\label{sep}
		X\stackrel{d}{=}Y_1+Y_2.
	\end{equation*}
\end{definition}

The first observation to make is that the restriction of $Y_i$ to be integer-valued is a non-trivial one. Since the sum
$Y_1+Y_2$ is integer-valued and the summands i.i.d., there are two possible cases: either $\Prob(Y_i\in\IZ)=1$ or $\Prob(Y_i\in\IZ+\frac12)=1$. To give a concrete example, let $Z\sim\mathrm{Bin}(2,\frac12)$. Then the integer-valued random variable $X=Z-1$, with probability mass function $p_{-1}=p_1=\frac14,\ p_0=\frac12$ can be written as the sum
of two independent $\mathrm{unif}\{-\frac12,\frac12\}$ random variables but not as the sum of two i.i.d.\ integer-valued
random variables.

Further, it is worth mentioning that this property (besides the restriction to integer-valued summands) is a lot weaker than the one of being {\em infinitely divisible} as the latter requires a decomposition into not only two, but arbitrarily many summands, i.e.\ that for any $n\in\IN$ there exist i.i.d.\ $Y_1,\dots,Y_n$ such that $X$ has the same distribution as the sum  $Y_1+\ldots+Y_n$.

\begin{corollary}\label{asf}
	For fixed $k\in\IN_0$, consider the random walk $(S_n)_{n\in\IN_0}$ on $\IZ$, defined by $S_0=k,\ S_{n+1}=S_n+X_{n+1}$, where $(X_n)_{n\in\mathbb{N}}$ is an i.i.d.\ sequence of integer-valued random variables with $\IE(X_1)>0$ and $\Prob(X_1\geq-2)=1$. If the marginal increment distribution $\mathcal{L}(X_1)$ is separable, i.e.\ there exist i.i.d.\ random variables $Y_1,Y_2$ such that
	\[X_1\stackrel{d}{=}Y_1+Y_2,\] the probability of the random walk $(S_n)_{n\in\mathbb{N}_0}$ to eventually become negative is given by
	\[\Prob(S_n<0\textup{ for some }n\in\IN)= \rho^{k+1}\,\bigg(1-\frac{\sigma}{2}\cdot\frac{1-(1-2\rho_\mathrm{odd})^{k+1}}{1-\tau\cdot(1-\tau_\mathrm{odd})}\bigg).
	\]
	The (conditional) probabilities $\sigma$, $\rho$ and $\tau$ as well as $\rho_\mathrm{odd}$ and $\tau_\mathrm{odd}$ appearing here refer to the ones introduced in \eqref{probs} corresponding to the left-continuous random walk on $\IZ$ starting at $0$ with increments distributed as $Y_1$.
\end{corollary}

\begin{proof}
	This result readily follows from Corollary \ref{generalstart}: Define the random walk $(\tilde{S}_n)_{n\in\IN_0}$ by $\tilde{S}_0=k$ and $\tilde{S}_n=\tilde{S}_{n-1}+Y_n$, $n\in\IN$, where $(Y_n)_{n\in\IN}$ is an i.i.d.\ sequence of copies of $Y_1$. Then $(\tilde{S}_n)_{n\in\IN_0}$ is a left-continuous, non-monotonic random walk on $\IZ$ with positive drift, to which Corollary \ref{generalstart} applies. The final observation needed to conclude is that
	$(\tilde{S}_{2n})_{n\in\IN_0}\stackrel{d}{=}(S_n)_{n\in\IN_0}$.
\end{proof}

\begin{remark}
	In the same way, one could treat a random walk on $\IZ$ with increment distribution being the convolution
	of $k\geq 3$ i.i.d.\ $\{-1\}\cup\IN_0$-valued random variables by introducing $k-1$ intermediate steps and analyze
	$(\tilde{S}_{kn})_{n\in\IN_0}$ instead. Lemma \ref{calc} and Theorem \ref{skipfree} still apply to the skip-free walk
	$(\tilde{S}_n)_{n\in\IN_0}$, however, the Markov chain and corresponding matrix calculations used in the proof
	of Theorem \ref{even} would become a lot harder to handle resp.\ to perform.
\end{remark}

\begin{example}
	Consider the random walk $(S_n)_{n\in\IN_0}$ with increments $X_n=S_n-S_{n-1}$, where $(X_n)_{n\in\IN}$ is an i.i.d.\ sequence and $X_n+2$ Poisson distributed with parameter $\lambda>2$. Then $X_n$ can be written as the sum of two independent	$\mathrm{Poi}(\frac\lambda2)$ random variables shifted to the left by 1.
	The auxiliary random walk (with intermediate steps) is therefore left-continuous and has increments as in the example at the end of Section \ref{sec3}. Choosing $\lambda=3$ for instance, we can therefore conclude by Corollary \ref{asf} for arbitrary $k\in\IN_0$:
	\[\Prob(S_n<0\textup{ for some }n\in\IN\,|\,S_0=k)\approx 0.417^{k+1}\,\bigg(1-\frac{0.156\cdot\big(1-(-0.413)^{k+1}\big)}{0.915}\bigg).\]
	Given a start at e.g.\ $0$ or $2$ this amounts to $\Prob(S_n<0\textup{ for some }n\in\IN\,|\,S_0=0)\approx0.317$ and $\Prob(S_n<0\textup{ for some }n\in\IN\,|\,S_0=2)\approx0.059$ respectively.
\end{example}

%\section*{About the author:}
%   We would like a short biographical sketch,
%   beyond just your affiliation to be placed
%   after the bibliography.
%   And below that, your full address.

\vspace{0.5cm}\noindent
	{\sc \small 
	Timo Vilkas\\
	Statistiska institutionen,\\
	Ekonomihögskolan vid Lunds universitet,\\
	220\,07 Lund, Sweden.}\\
	timo.vilkas@stat.lu.se\\


\begin{thebibliography}{9}

\bibitem{Athreya-Ney}
{\sc K.B.\ Athreya} and {\sc P.E.\ Ney} (1972):
{\em ``Branching Processes''},
Springer Berlin, Heidelberg.
\bibitem{skipfree}
{\sc M.\ Brown,} {\sc E.A.\ Pek\"oz} and {\sc S.M.\ Ross} (2010):
{Some results for skip-free random walk},
{\em Probability in the Engineering and Informational Sciences},
Vol.\ 24 (4), pp.\ 491-507.
\bibitem{GWT}
{\sc M.\ Dwass} (1969):
{The Total Progeny in a Branching Process and a Related Random Walk},
{\em Journal of Applied Probability},
Vol.\ 6 (3), pp.\ 682-686.
\bibitem{Durrett}
{\sc R.\ Durrett,}
{\em ``Essentials of Stochastic Processes (2nd edition)''},
Springer, 2012.
\bibitem{Markov}
{\sc C.M.\ Grinstead,} and {\sc J.L.\ Snell}: (2012) 
{\em ``Introduction to Probability (2nd edition)''},
American Mathematical Society.
\bibitem{hitting2}
{\sc R.\ van der Hofstad} and {\sc M.\ Keane} (2008):
{An Elementary Proof of the Hitting Time Theorem},
{\em The American Mathematical Monthly},
Vol.\ 115 (8), pp.\ 753-756.
\bibitem{skipfree2}
{\sc G.\ Letac,} {\sc P.\ Mazet} and {\sc G.\ Schiffmann} (1976):
{A Note on Left-Continuous Random Walks},
{\em Journal of Applied Probability},
Vol.\ 13 (4), pp.\ 814-817.
\bibitem{hitting}
{\sc A.G.\ Pakes} (1973):
{Conditional limit theorems for a left-continuous random walk},
{\em Journal of Applied Probability},
Vol.\ 10 (1), pp.\ 39-53.
\bibitem{Spitzer}
{\sc F.\ Spitzer,}
{\em ``Principles of Random Walk (2nd edition)''},
Springer New York, 1964.

\end{thebibliography}
\end{document}